\documentclass[12pt,reqno]{amsart}
\makeatletter
  {}{}
\usepackage{amsmath,amssymb,amsthm}
\usepackage{color}
\usepackage{url}
\newtheorem{theorem}{Theorem}[section]

\newtheorem{lemma}{Lemma}[section]
\newtheorem{remark}{Remark}[section]
\newtheorem{example}{Example}[section]
\numberwithin{equation}{section}
\usepackage[colorlinks=true]{hyperref} 
\hypersetup{urlcolor=blue, citecolor=red , linkcolor= blue}
  \newtheorem{thqt}{Theorem}

\begin{document}


\title[Joint numerical radius of spherical Aluthge transforms]
{Joint numerical radius of spherical Aluthge transforms of tuples of Hilbert space operators}

\author[K. Feki]{Kais Feki}
\address{University of Sfax, Sfax, Tunisia.}
\email{kais.feki@hotmail.com}

\author[T. Yamazaki]{Takeaki Yamazaki}
\address{Toyo University, Saitama, Japan.}
\email{t-yamazaki@toyo.jp}

\begin{abstract}
Let $\mathbf{T}=(T_1,\ldots,T_d)$ be a $d$-tuple of operators on a complex Hilbert space $\mathcal{H}$. The spherical Aluthge transform of $\mathbf{T}$ is the $d$-tuple given by $\widehat{\mathbf{T}}:=(\sqrt{P}V_1\sqrt{P},\ldots,\sqrt{P}V_d\sqrt{P})$ where $P:=\sqrt{T_1^*T_1+\ldots+T_d^*T_d}$ and $(V_1,\ldots,V_d)$ is a joint partial isometry such that $T_k=V_k P$ for all $1 \le k \le d$. In this paper, we prove several inequalities involving the joint numerical radius and the joint operator norm of $\widehat{\mathbf{T}}$. Moreover, a characterization of the joint spectral radius of an operator tuple $\mathbf{T}$ via $n$-th iterated of spherical Aluthge transform is established.
\end{abstract}
\maketitle

\section{Introduction and Preliminaries}\label{s1}
Throughout this paper, $\mathcal{H}$ will be a complex Hilbert space, with the inner product $\langle\cdot\mid \cdot\rangle$ and the norm $\|\cdot\|$.  $\mathcal{B}(\mathcal{H})$ stands for the Banach algebra of all bounded linear operators on $\mathcal{H}$ and $I$ denotes the identity operator on $\mathcal{H}$. In all that follows, by an operator we mean a bounded linear operator. The range and the null space of an operator $T$ are denoted by ${\mathcal R}(T)$ and ${\mathcal N}(T)$, respectively. Also, $T^*$ will be denoted to be the adjoint of $T$. An operator $T$ is called positive if $\langle Tx\mid x\rangle\geq0$ for all $x\in{\mathcal H }$, and we then write $T\geq 0$. Further, the square root of every positive operator $T$ is denoted by
$T^{\frac{1}{2}}$. If $T\geq 0$, then the absolute value of $T$ is denoted by $|T|$ and given by $|T|=(T^*T)^{\frac{1}{2}}$.

For $T \in \mathcal{B}(\mathcal{H})$, the spectral radius of $T$ is defined by
\[r(T)=\sup \left\{ |\lambda|\,;\; \lambda \in \sigma(T)\right \},\]
where $\sigma(T)$ denotes the spectrum  of $T$. Moreover, the numerical radius and operator norm of $T$ are denoted by $\omega \left( T \right)$ and $\left\| T \right\|$ respectively and they are given by
\begin{equation*}
\omega \left( T \right)=\sup \left\{ \left| \left\langle Tx\mid x \right\rangle  \right|\,;\;\text{ }x\in \mathcal{H},\left\| x \right\|=1 \right\}
\end{equation*}
and
\begin{equation*}
\left\| T \right\|=\sup \left\{ \left\| Tx \right\|\,;\text{ }x\in \mathcal{H},\left\| x \right\|=1 \right\}.
\end{equation*}
It is well-known that for $T \in \mathcal{B}(\mathcal{H})$ we have
\begin{align}\label{1.1}
\frac{\|T\|}{2} \leq \max \left \{r(T), \frac{\|T\|}{2} \right\}\leq \omega(T)\leq \|T\|.
\end{align}
It has been shown in \cite{Y2007} that if $T\in\mathcal{B}(\mathcal{H})$, then
\begin{align}\label{zamnum}
\omega(T) = \displaystyle{\sup_{\theta \in \mathbb{R}}}{\left\|\Re(e^{i\theta}T)\right\|},
\end{align}
where $\Re(X):=\frac{X+X^{*}}{2}$ for a given operator $X$. For more results, we refer the reader to the book by Gustafson and Rao \cite{gusta}.

An operator $U\in \mathcal{B}(\mathcal{H})$ is said to be a partial isometry $\|Ux\|=\|x\|$ for every $x\in \mathcal{N}(A)^\perp$. Let $T=U|T|$ be the polar decomposition of $T\in \mathcal{B}(\mathcal{H})$ with $U$ is a partial isometry. The Aluthge transform of $T$ was first defined
in \cite{Alu} by $\widetilde{T}:=|T|^{\frac{1}{2}}U|T|^{\frac{1}{2}}$. This transformation has attracted considerable attention over the last two decades (see, for example, \cite{Ando,CJL,DySc,JKP,JKP2,LLY,Yam}).  The following properties of $\widetilde{T}$ are well-known (see \cite{JKP}):
\begin{enumerate}
  \item [(i)] $\|\widetilde{T}\| \leq \|T\|$,
  \item [(ii)] $r(\widetilde{T})= r(T)$,
  \item [(iii)] $\omega(\widetilde{T}) \leq \omega(T)$.
\end{enumerate}

Let $\mathbf{T} = (T_1,\ldots,T_d)\in \mathcal{B}(\mathcal{H })^d$ be a $d$-tuple of operators. The joint numerical range of $\mathbf{T}$ is introduced by A.T. Dash \cite{dach} as:
$$JtW(\mathbf{T})=\left\{(\langle T_1 x\mid x\rangle,\ldots,\langle T_d x \mid x\rangle)\,;\,x \in \mathcal{H},\;\|x\|=1\right\}.$$
If $d=1$, we get the definition of the classical numerical range of an operator $T$, denoted by $W(T)$, which is firstly introduced by Toeplitz in \cite{t1}. It is well-known that $W(T)$ is convex (see \cite{li,gus}). Unlike the classical numerical range, $JtW(\mathbf{T})$ may be non convex for $d\geq2$. For a survey of results concerning the convexity of $JtW(\mathbf{T})$, the reader may see \cite{dach,lipoon2} and their references. The joint numerical radius of an operator tuple $\mathbf{T}=(T_1,\ldots,T_d)$ is defined in \cite{cmt} as
\begin{align*}
\omega(\mathbf{T})
&=\displaystyle\sup\left\{\|\lambda\|_2\,; \lambda=(\lambda_1,\ldots,\lambda_d)\in
JtW ( \mathbf{T} ) \right\}\\
&=\displaystyle\sup\left\{\left(\displaystyle\sum_{k=1}^d|\langle T_kx\mid x\rangle|^2\right)^{\frac{1}{2}};\;x\in \mathcal{H},\;\|x\|=1\right\}.
\end{align*}
It was shown in \cite{bakfeki02} that for an operator tuple $\mathbf{T}=(T_1,\ldots,T_d)\in \mathcal{B}(\mathcal{H})^d$, we have
\begin{equation}\label{newnumrad}
\omega(\mathbf{T})= \displaystyle\sup_{(\lambda_1,\ldots,\lambda_d)\in \overline{\mathbb{B}}_d}\omega(\lambda_1T_1+\ldots+\lambda_dT_d),
\end{equation}
where $\mathbb{B}_d$ denotes the
open unit ball in $\mathbb{C}^d$ with respect to the euclidean norm, and
$\overline{\mathbb{B}}_d$ is its closure
i.e.
$$\overline{\mathbb{B}}_d:=\left\{\lambda=(\lambda_1,\ldots,\lambda_d)\in \mathbb{C}^d\,;\;\|\lambda\|_2^2:=\sum_{k=1}^d|\lambda_k|^2\leq1 \right\} .$$

Given a $d$-tuple $\mathbf{T}= (T_1,\ldots,T_d)$ of operators on $\mathcal{H}$, the joint norm of $\mathbf{T}$ is defined as
$$\|\mathbf{T}\|:=\displaystyle\sup\left\{\left(\displaystyle\sum_{k=1}^d\|T_kx\|^2\right)^{\frac{1}{2}};\;x\in \mathcal{H},\;\|x\|=1\right\}.$$ 
Notice that $\|\cdot\|$ and $\omega(\cdot)$ are equivalent norms on $\mathcal{B}(\mathcal{H})^d$. More precisely, for every $\mathbf{T}= (T_1,\ldots,T_d)\in \mathcal{B}(\mathcal{H})^d$ we have
\begin{equation}\label{eqcor2.3}
\frac{1}{2\sqrt{d}}\|\mathbf{T}\|\leq\omega(\mathbf{T})\leq \|\mathbf{T}\|.
\end{equation}
Moreover, the inequalities in \eqref{eqcor2.3} are sharp (see \cite{bakfeki01,popescu01}).

Let $\mathbf{T}=(T_1,\ldots,T_d)\in \mathcal{B}(\mathcal{H})^d$ be a $d$-tuple of operators, and consider $S=\begin{pmatrix}
T_{1} \\
\vdots \\
T_{d}%
\end{pmatrix}$ as an operator from $\mathcal{H}$ into $\mathbb{H}:=\oplus_{i=1}^d\mathcal{H}$, that is,
\begin{align}
S=\begin{pmatrix}
T_{1} \\
\vdots \\
T_{d}%
\end{pmatrix} \colon \mathcal{H} & \rightarrow\mathbb{H},\;x\mapsto
{}^{t}(T_1x,\ldots,T_dx).  \label{setting 1}
\end{align}%
Then, we have $S^{\ast }S=(T_{1}^{\ast },\ldots,T_{d}^{\ast})
\begin{pmatrix}
T_{1} \\
\vdots \\
T_{d}%
\end{pmatrix}=\displaystyle\sum_{k=1}^dT_{k}^{\ast }T_{k}$. Since $S$ is an operator from $\mathcal{H}$ into $\mathbb{H}$, then $S$ has a classical polar decomposition $S= V P$, that is,
\begin{equation*}
\begin{pmatrix}
T_{1} \\
\vdots \\
T_{d}%
\end{pmatrix}=\begin{pmatrix}
V_{1} \\
\vdots \\
V_{d}%
\end{pmatrix}P=\begin{pmatrix}
V_{1}P \\
\vdots \\
V_{d}P%
\end{pmatrix}, \label{setting 2}
\end{equation*}%
where $V=\begin{pmatrix}
V_{1} \\
\vdots \\
V_{d}%
\end{pmatrix}$ is a partial isometry from $\mathcal{H}$ to $\mathbb{H}$ and $P$ is the positive operator on $\mathcal{H}$ given by
$$P=(S^*S)^{\frac{1}{2}}=\sqrt{T_{1}^{\ast }T_{1}+\ldots+T_{d}^{\ast }T_{d}}.$$
So $%
R:=V^*V=(V_{1}^{\ast },\ldots, V_{d}^{\ast })\begin{pmatrix}
V_{1} \\
\vdots \\
V_{d}%
\end{pmatrix}=\displaystyle\sum_{k=1}^dV_{k}^{\ast }V_{k}$
is the orthogonal projection onto the initial space of $V$ which is
\begin{equation}
\bigcap_{i=1}^d \mathcal{N}(T_{i})
=\mathcal{N}(S)
=\mathcal{N}(P) =\bigcap_{i=1}^d \mathcal{N}(V_{i}).
\label{eq:kernel condition}
\end{equation}%

For $\mathbf{T}=(T_{1},\ldots,T_{d})\in\mathcal{B}(\mathcal{H})^d$, the spherical Aluthge transform of $\mathbf{T}$ is defined as
\begin{equation*}
\widehat{\mathbf{T}}=(\widehat{T_{1}},\ldots,\widehat{T_{d}}):=\left( \sqrt{P}V_{1}\sqrt{P},\ldots,\sqrt{P}V_{d}\sqrt{P}\right) \text{ (cf. \cite{CuYo7}, \cite{CuYo6}, \cite{KiYo1}).}
\label{Def-Alu1}
\end{equation*}
This transformation has been recently investigated by C. Benhida et al. in \cite{benhida}. It should be mention here that $\widehat{T_{i}}=\sqrt{P}V_{i}\sqrt{P}$ is not the Aluthge transform of $T_i$ (for $i\in\{1,\ldots,d\}$). 
When the operators $T_k$ are pairwise commuting, we say that $\mathbf{T}$ is a commuting $d$-tuple.


Let $\mathbf{T}= (T_1,\ldots,T_d)\in \mathcal{B}(\mathcal{H})^d$ be a commuting $d$-tuple of operators.
There are several different notions of a spectrum. For a good description, the reader is referred to \cite{curto} and the references there in. There is a well-known notion of a joint spectrum of a commuting $d$-tuple $\mathbf{T}$ called the Taylor joint spectrum denoted by $\sigma_T(\mathbf{T})$ (see \cite{tay}).
It is shown in \cite{benhida} that
$\sigma_{T}(\widetilde{\mathbf{T}})=
\sigma_{T}(\mathbf{T})$ for commuting
$\mathbf{T}\in \mathcal{B(H)}^{d}$.
The joint spectral radius of $\mathbf{T}$ is defined to be the number
\begin{equation*}
r(\mathbf{T})=\sup\{\|\lambda\|_2\,;\lambda=(\lambda_1,\ldots,\lambda_d) \in \sigma_T(\mathbf{T})\}.
 \end{equation*}
It should be mention here that Ch\=o and \`Zelazko proved in \cite{cz} that this definition of $r(\mathbf{T})$ is independent of the choice of the joint spectrum of $\mathbf{T}$. Furthermore, an analogue of the Gelfand-Beurling spectral radius formula for single operators has been established by M\"uller and Soltysiak in \cite{mus} for commuting tuples.
Let $\mathbf{T}=(T_{1},\ldots,T_{m})\in
\mathcal{B}(\mathcal{H})^{m}$ and
$\mathbf{S}=(S_{1},\ldots,S_{n})\in \mathcal{B}(\mathcal{H})^{n}$.
Then the product $\mathbf{T}\mathbf{S}$ is
defined by
$$ \mathbf{T}\mathbf{S}=
(T_{1}S_{1},\ldots,T_{1}S_{n}, T_{2}S_{1},\ldots,T_{2}S_{n},\ldots,
T_{m}S_{1},\ldots,T_{m}S_{n})\in \mathcal{B}(\mathcal{H})^{mn}. $$
Especially, $\mathbf{T}^{2}=\mathbf{T}\mathbf{T}$
and $ \mathbf{T}^{n+1}=\mathbf{T}\mathbf{T}^{n}$.
It was shown in \cite{mus} (cf. \cite{bun}) that
if $\mathbf{T}$ is commuting, then
the joint spectral radius of $\mathbf{T}$ is given by
\begin{equation}\label{jointradii002}
r(\mathbf{T})=\lim_{n\to \infty}
\| \mathbf{T}^{n}\|^{\frac{1}{n}}.
\end{equation}
%
In this paper, we shall show several inequalities for
spherical Aluthge transform which are known in the
single operator case in Sections \ref{s2} and \ref{s3}.
Then, in Section \ref{s4} we shall show a characterization of
joint spectral radius via $n$-th iterated of
spherical Aluthge transform.
It is an extension of the formula
$\displaystyle \lim_{n\to \infty}\|\widetilde{T}_{n}\|
=r(T)$, where $\widetilde{T}_{n}$ means
the $n$-th iterated of Aluthge transform
of a single operator shown in \cite{Yam}.

\section{Basic inequalities}\label{s2}
In this section, we present basic inequalities
for spherical Aluthge transform.

\begin{theorem}\label{thm:norm decreasing}
Let $\mathbf{T}=(T_{1},\ldots,T_{d})\in \mathcal{B}(\mathcal{H})^{d}$. Then,
$$\|\widehat{\mathbf{T}}\|\leq \|\mathbf{T}\|.$$
\end{theorem}

In order to prove our first result, we need the following lemmas.
\begin{lemma}\label{lem:norm of tuple of operators}
Let $\mathbf{T}=(T_{1},\ldots, T_{d})\in \mathcal{B}(\mathcal{H})^{d}$. Then
$$ \|\mathbf{T}\|=\left\|\sum_{k=1}^{d}T_{k}^{*}T_{k}\right\|^{\frac{1}{2}}.$$
%
\end{lemma}

\begin{proof}
Since $\sum_{k=1}^{d}T^{*}_{k}T_{k}\geq 0$, then it follows that
\begin{align*}
\|\mathbf{T}\|^2
& =
\sup_{\|x\|=1}\sum_{k=1}^{d}\|T_{k}x\|^{2}
 =
\sup_{\|x\|=1}\langle \sum_{k=1}^{d}T^{*}_{k}T_{k}x\mid x\rangle
 =
\left\|\sum_{k=1}^{d}T^{*}_{k}T_{k}\right\|.
\end{align*}
%
%
\end{proof}

\begin{lemma}\label{lem:Lemma 2}
Let $A, X_{k}\in \mathcal{B}(\mathcal{H})$ for $k=1,2,\ldots,d$.
Then
$$ \left\| \sum_{k=1}^{d}X_{k}^{*}AX_{k}\right\|
\leq
\left\| \sum_{k=1}^{d}X_{k}^{*}X_{k}\right\| \|A\|. $$
\end{lemma}

\begin{proof}
It can be seen that
\begin{align*}
\left\| \sum_{k=1}^{d}X_{k}^{*}AX_{k}\right\|
 & =\left\| \begin{pmatrix}
X_{1}^{*} & \cdots & X_{d}^{*} \\
0 & \cdots & 0 \\
\vdots & & \vdots \\
0 & \cdots & 0
\end{pmatrix}
\begin{pmatrix}
A &   \\
 & \ddots  \\
 &  & A
\end{pmatrix}
\begin{pmatrix}
X_{1} & 0 & \cdots & 0 \\
\vdots & \vdots &  & \vdots \\
X_{d} & 0 & \cdots & 0
\end{pmatrix}\right\| \\
& \leq
\left\|\begin{pmatrix}
A &   \\
 & \ddots  \\
 &  & A
\end{pmatrix}\right\|\left\| \begin{pmatrix}
X_{1} & 0 & \cdots & 0 \\
\vdots & \vdots &  & \vdots \\
X_{d} & 0 & \cdots & 0
\end{pmatrix}\right\|^{2}
\\
&=\|A\|
\left\|
\begin{pmatrix}
X_{1}^{*} & \cdots & X_{d}^{*} \\
0 & \cdots & 0 \\
\vdots & & \vdots \\
0 & \cdots & 0
\end{pmatrix}
\begin{pmatrix}
X_{1} & 0 & \cdots & 0 \\
\vdots & \vdots &  & \vdots \\
X_{d} & 0 & \cdots & 0
\end{pmatrix}\right\|\\
& =\|A\|\left\| \sum_{k=1}^{d}X_{k}^{*}X_{k}\right\|.
\end{align*}
This proves the desired inequality.
\end{proof}

\begin{proof}
[Proof of Theorem \ref{thm:norm decreasing}]
First of all, we notice that
$$ \|\mathbf{T}\|^{2}=\left\| \sum_{k=1}^{d}T_{k}^{*}T_{k}\right\|=
\|P\|^{2} $$
by Lemma \ref{lem:norm of tuple of operators}.
Then by using Lemma \ref{lem:Lemma 2}, we have
\begin{align*}
\|\widehat{\mathbf{T}}\|^{2} & =
\left\| \sum_{k=1}^{d} \widehat{T}_{k}^{*}
\widehat{T}_{k}
\right\| \\
& =
\left\| \sum_{k=1}^{d} P^{\frac{1}{2}}V_{k}^{*}
PV_{k}P^{\frac{1}{2}}\right\| \\
& \leq
\|P\|
\left\| \sum_{k=1}^{d} P^{\frac{1}{2}}V_{k}^{*}
V_{k}P^{\frac{1}{2}}\right\|
 =
\|P\|\cdot \|P\| =\|\mathbf{T}\|^{2},
\end{align*}
where the third equation follows from the fact that $\sum_{k=1}^{d}V_{k}^{*}V_{k}$ is a projection onto $\overline{\mathcal{R}(P)}$.
%
\end{proof}

Next, we shall show inequalities of joint numerical radius for spherical Aluthge transform. This discussion will be divided into two parts. We treat non-commuting tuples of operators in the first part.

\begin{theorem}
\label{thm:first step on numerical radius}
Let $\mathbf{T}=(T_{1},\ldots,T_{d})\in \mathcal{B}(\mathcal{H})^{d}$. Then,
$$\omega(\widehat{\mathbf{T}})\leq
\frac{1}{2}\omega(\mathbf{T})+\frac{1}{2}
\omega(\mathbf{T}_{1}),$$
where $\mathbf{T}_{1}:=(PV_{1},\ldots, PV_{d})$.
\end{theorem}

To prove the result, we will use the following theorems.
\begin{thqt}[\cite{H1965, SW1968}]
\label{thmqt: characterization of numerical range}
Let $T\in \mathcal{B}(\mathcal{H})$. Then
$$ \overline{W(T)}=\bigcap_{\mu\in \mathbb{C}}
\{\lambda\in \mathbb{C}\; ;\,
|\lambda-\mu|\leq \|T-\mu I\|\}.$$
\end{thqt}

\begin{thqt}[\cite{CPR1990},
{\cite[Theorem 3.12.1]{F2001}}]
\label{thmqt: Heinz inequality}
Let $A$ be a self-adjoint invertible operator and
$X\in \mathcal{B}(\mathcal{H})$. Then
$$2\|X\|\leq \|AXA^{-1}+A^{-1}XA\|.$$
\end{thqt}

\begin{proof}
[Proof of Theorem \ref{thm:first step on numerical radius}]
In view of \eqref{newnumrad}, we have
\begin{equation}
\omega(\mathbf{T})=\sup_{(\lambda_{1},\ldots,\lambda_{d})\in \overline{\mathbb{B}}_{d}}
\omega(\lambda_{1}T_{1}+\ldots+\lambda_{d}T_{d})=
\sup_{(\lambda_{1},\ldots,\lambda_{d})\in \overline{\mathbb{B}}_{d}}\omega(U_{\lambda}P),
\label{eq:numerical radius of UP}
\end{equation}
\begin{equation}
\omega(\widehat{\mathbf{T}})=\sup_{(\lambda_{1},\ldots,\lambda_{d})\in \overline{\mathbb{B}}_{d}}
\omega(P^{\frac{1}{2}}U_{\lambda}
P^{\frac{1}{2}}) \text{ and }
\omega(\mathbf{T}_{1})=\sup_{(\lambda_{1},\ldots,\lambda_{d})\in \overline{\mathbb{B}}_{d}}
\omega(PU_{\lambda}),
\label{eq:numerical radius of PU}
\end{equation}
where $U_{\lambda}=\lambda_{1}V_{1}+\ldots +\lambda_{d}V_{d}$.
We shall prove
$$ \overline{W(P^{\frac{1}{2}}U_{\lambda}
P^{\frac{1}{2}})}
\subseteq
\overline{W\left(\frac{U_{\lambda}P+
PU_{\lambda}}{2}\right)}, $$
where $\overline{W(X)}$ means the closure of numerical range of $X\in \mathcal{B}(\mathcal{H})$.
By taking into consideration Theorem \ref{thmqt: characterization of numerical range}, it suffices to prove the following norm inequality.
\begin{equation}
 \| P^{\frac{1}{2}}U_{\lambda}P^{\frac{1}{2}}-\mu I\|
\leq
\left\| \frac{U_{\lambda}P+
PU_{\lambda}}{2}-\mu I\right\|
\label{eq: norm inequality}
\end{equation}
for all $\mu\in \mathbb{C}$.

For $\varepsilon >0$, let $P_{\varepsilon}:=
P+\varepsilon I>0$. Then $P_{\varepsilon}$ is
positive invertible. Then by Theorem
\ref{thmqt: Heinz inequality}, we have
\begin{align*}
2\| P_{\varepsilon}^{\frac{1}{2}}U_{\lambda}
P_{\varepsilon}^{\frac{1}{2}}-\mu I\|
& \leq
\| P_{\varepsilon}^{\frac{1}{2}}
(P_{\varepsilon}^{\frac{1}{2}}U_{\lambda}
P_{\varepsilon}^{\frac{1}{2}}-\mu I)
P_{\varepsilon}^{-\frac{1}{2}}
+
P_{\varepsilon}^{-\frac{1}{2}}
(P_{\varepsilon}^{\frac{1}{2}}U_{\lambda}
P_{\varepsilon}^{\frac{1}{2}}-\mu I)
P_{\varepsilon}^{\frac{1}{2}}\|\\
& =
\| P_{\varepsilon} U_{\lambda}+
U_{\lambda}P_{\varepsilon}-2\mu I\|.
\end{align*}
By letting $\varepsilon \searrow 0$, we get \eqref{eq: norm inequality}, and hence
$$
\overline{W(P^{\frac{1}{2}}U_{\lambda}P^{\frac{1}{2}})}
 \subseteq
\overline{W\left(\frac{U_{\lambda}P+
PU_{\lambda}}{2}\right)}
 \subseteq
\frac{1}{2} \left\{
\overline{W(PU_{\lambda})}+
\overline{W(U_{\lambda}P)}\right\}.
$$
Therefore, we get
$$ \omega(P^{\frac{1}{2}}U_{\lambda}
P^{\frac{1}{2}})\leq \frac{1}{2}
\Big(\omega(PU_{\lambda})+\omega(U_{\lambda}P)\Big),$$
which in turn implies, by taking the supremum over all $(\lambda_{1},\ldots,\lambda_{d})\in \overline{\mathbb{B}}_{d}$, that
\begin{equation}\label{marj00}
\omega(\widehat{\mathbf{T}})\leq
\frac{1}{2}\omega(\mathbf{T})+
\frac{1}{2}\omega(\mathbf{T}_{1}).
\end{equation}
\end{proof}

The second part of this discussion,
we shall treat commuting tuples of operators.

\begin{theorem}\label{thm: commuting case}
Let $\mathbf{T}=(T_{1},\ldots, T_{d})\in
\mathcal{B}(\mathcal{H})^{d}$
be a commuting tuple of operators.
Then
$$\omega(\widehat{\mathbf{T}})
\leq\omega(\mathbf{T}).$$
\end{theorem}

To prove this, we will introduce the following
lemma.

\begin{lemma}\label{lem: commuting}
Let $\mathbf{T}=(T_{1},\ldots, T_{d})\in
\mathcal{B}(\mathcal{H})^{d}$,
and let $T_{j}=V_{j}P$ with
$P=(\sum_{j=1}^{d}T_{j}^{*}T_{j})^{\frac{1}{2}}$. Then
$\mathbf{T}$ is commuting if and only if
$$ V_{j}PV_{k}=V_{k}PV_{j} $$
holds for $j,k=1,\ldots,d$.
\end{lemma}

\begin{proof}
Since $T_{j}T_{k}=T_{k}T_{j}$, we have
$V_{j}PV_{k}P=V_{k}PV_{j}P$, that is,
$V_{j}PV_{k}=V_{k}PV_{j}$ holds on
$\overline{\mathcal{R}(P)}$.
By \eqref{eq:kernel condition},
$\displaystyle
\overline{\mathcal{R}(P)}^{\perp} =
\mathcal{N}(P)=
\bigcap_{k=1}^{d} \mathcal{N}(V_{k})\subset
\mathcal{N}(V_{k})$ for $k=1,\ldots,d$.
Hence we have
$V_{j}PV_{k}=V_{k}PV_{j}=0$ on $\mathcal{N}(P)$.
Therefore $V_{j}PV_{k}=V_{k}PV_{j}$ holds
on $\mathcal{H}=\overline{\mathcal{R}(P)}\oplus
\mathcal{N}(P)$.
The converse implication is obvious.
Thus the proof is completed.
\end{proof}

\begin{proof}
[Proof of Theorem \ref{thm: commuting case}]
Since Theorem \ref{thm:first step on numerical radius},
we have only to prove the following
inequality.
$$
\omega(\mathbf{T}_{1}) \leq \omega(\mathbf{T}),
$$
that is, we will prove
\begin{equation}
\omega(PU_{\lambda}) \leq \omega(U_{\lambda}P).
\label{eq: problem}
\end{equation}
%

%
%
%
Since $\sum_{k=1}^{d}V_{k}^{*}V_{k}$ is a projection
onto $\overline{\mathcal{R}(P)}$, we have
$$
\langle PU_{\lambda}x|x \rangle
 =
\langle \left(\sum_{k=1}^{d}V_{k}^{*}V_{k}\right)
PU_{\lambda}x|x \rangle \\
 =
\sum_{k=1}^{d} \langle
V_{k}PU_{\lambda} x| V_{k} x
\rangle.$$
Moreover by Lemma \ref{lem: commuting},
\begin{align*}
V_{k}PU_{\lambda}
& =
V_{k}P
\left(\sum_{j=1}^{d}\lambda_{j}V_{j}\right)
 =
\left(\sum_{j=1}^{d}\lambda_{j}V_{j}\right)PV_{k}
=
U_{\lambda}PV_{k}.
\end{align*}
Then
$$ \langle PU_{\lambda}x|x \rangle  =
\sum_{k=1}^{d}\langle V_{k}P U_{\lambda}x|
V_{k}x\rangle
=
\sum_{k=1}^{d}\langle U_{\lambda}P V_{k}x|
V_{k}x\rangle.
$$
Put $y_{k}=\frac{V_{k}x}{\|V_{k}x\|}$.
Since $\sum_{k=1}^{d}V_{k}^{*}V_{k}$ is a projection
onto $\overline{\mathcal{R}(P)}$, we have
\begin{align*}
|\langle PU_{\lambda}x|x\rangle |
& =
\left| \sum_{k=1}^{d} \| V_{k}x \|^{2}
\langle U_{\lambda}P y_{k} \mid y_{k}\rangle \right| \\
& \leq
 \sum_{k=1}^{d} \| V_{k}x \|^{2}
\left|\langle U_{\lambda}P y_{k} \mid y_{k}\rangle \right| \\
& \leq
 \sum_{k=1}^{d} \| V_{k}x \|^{2}
\omega(U_{\lambda}P) \\
& =
\langle \sum_{k=1}^{d} V_{k}^{*}V_{k}x \mid x \rangle\,
\omega(U_{\lambda}P)
 \leq
\omega (U_{\lambda}P).
\end{align*}
Therefore we have
$\omega(PU_{\lambda}) \leq \omega(U_{\lambda}P)$,
and the proof is finished.
\end{proof}

\begin{remark}
We may not remove the commutative condition
in Theorem \ref{thm: commuting case}. In fact
there is a counterexample of
non-commuting tuple of operators
for \eqref{eq: problem} as follows.
\end{remark}

\begin{example}
Let
$$\mathbf{T}=(T_{1}, T_{2})=(
\begin{pmatrix}0&1\\0&0\end{pmatrix},
\begin{pmatrix}0&-1\\1&0\end{pmatrix}). $$
%
A short calculation shows that
$$\omega^2(T_1,T_2)=\sup_{|x|^2+|y|^2=1}\left\{ |x\overline{y}|^2+4|\Im(x\overline{y})|^2\right\},$$
where $\Im (x\overline{y})$ means the imaginary
part of $x\overline{y}$.
Let $(x,y)\in \mathbb{C}^2$ be such that $|x|^2+|y|^2=1$. Then, it can be observed that
$$|x\overline{y}|^2+4|\Im(x\overline{y})|^2\leq \frac{5}{4}.$$
This implies that $\omega^2(T_1,T_2)\leq \frac{5}{4}.$ Moreover, if we take $x=\frac{1+i}{2}$ and $y=\overline{x}$ we can prove that
$$\omega^2(T_1,T_2)= \frac{5}{4}.$$

On the other hand it is not difficult to verify that
$$P=\begin{pmatrix}1&0\\0&\sqrt{2}\end{pmatrix},\;V_1=\begin{pmatrix}0&\tfrac{\sqrt{2}}{2}\\0&0\end{pmatrix}\;\text{ and }\;V_2=\begin{pmatrix}0&-\tfrac{\sqrt{2}}{2}\\1&0\end{pmatrix}.$$
So, we get
$$PV_1=\begin{pmatrix}0&\tfrac{\sqrt{2}}{2}\\ 0&0\end{pmatrix}=\frac{\sqrt{2}}{2}S_1\;\text{ and }\;PV_2=\begin{pmatrix}0&-\tfrac{\sqrt{2}}{2}\\ \sqrt{2}&0\end{pmatrix}=\frac{\sqrt{2}}{2}S_2,$$
where $S_1=\begin{pmatrix}0&1\\0&0\end{pmatrix}\;\text{ and }\;S_2=\begin{pmatrix}0&-1\\2&0\end{pmatrix}.$ So, $\omega(PV_1,PV_2)=\tfrac{\sqrt{2}}{2}\omega(S_1,S_2)$. Now, it can be observed that
$$\omega^2(S_1,S_2)=\sup_{|x|^2+|y|^2=1}\left\{ |x\overline{y}|^2+|2i\Im(x\overline{y})+x\overline{y}|^2\right\}.$$
By similar arguments as above, on can verify that $\omega^2(S_1,S_2)=\frac{10}{4}$ which implies that
$$\omega(PV_1,PV_2)=\frac{5\sqrt{2}}{4}> \omega(T_1,T_2).$$
\end{example}

\section{Precise estimation of joint numerical radius}\label{s3}
In this section, we shall give a precise
estimation of joint numerical radius.

\begin{theorem}\label{thm:1/2}
Let $\mathbf{T}=(T_{1},\ldots,T_{d})\in \mathcal{B}(\mathcal{H})^{d}$ be a $d$-tuple of operators. Then,
$$ \omega(\mathbf{T}) \leq
\frac{1}{2}\|\mathbf{T}\|+
\frac{1}{2}\omega(\widehat{\mathbf{T}}).$$
\end{theorem}

\begin{remark}
By letting $d=1$ in Theorem \ref{thm:1/2}, we get the well-known result proved by the second author in \cite{Y2007} asserting that
$$ \omega(T) \leq \frac{1}{2}\|T\|+\frac{1}{2}
\omega(\widetilde{T}),$$
for every $T\in \mathcal{B}(\mathcal{H})$.
\end{remark}

\begin{proof}
By \eqref{eq:numerical radius of UP},
we see that
$$\omega(\mathbf{T})
 = \sup_{(\lambda_{1},\ldots,\lambda_{d})\in \overline{\mathbb{B}}_d}
\omega(U_{\lambda} P),
$$

Now, let $x\in \mathcal{H}$ with $\|x\|=1$.
By the generalized polarization identity (see \cite{Y2007}) we get
\begin{align*}
\langle e^{i\theta}U_{\lambda} P x\mid x\rangle & =
\langle e^{i\theta}Px\mid U^{*}_{\lambda}x\rangle \\
& =\frac{1}{4}\bigl(\langle P(e^{i\theta}+U^{*}_{\lambda})x\mid  (e^{i\theta}+U^{*}_{\lambda})x\rangle
- \langle P(e^{i\theta}-U^{*}_{\lambda})x\mid  (e^{i\theta}-U^{*}_{\lambda})x\rangle \bigr)\\
&\;+\frac{i}{4}\bigl(\langle P(e^{i\theta}+iU^{*}_{\lambda})x\mid  (e^{i\theta}+iU^{*}_{\lambda})x
\rangle -\langle P(e^{i\theta}-iU^{*}_{\lambda})x\mid (e^{i\theta}-iU^{*}_{\lambda})x\rangle \bigr).
\end{align*}
Noting that all inner products of the terminal side are all positive since $P\geq0$. Hence, one observes that
\begin{align*}
\langle \Re(e^{i\theta}U_{\lambda}P)x\mid x\rangle
& = \Re(\langle e^{i\theta}U_{\lambda}Px\mid x\rangle)\\
& =\frac{1}{4}\bigl(\langle (e^{i\theta}+U^{*}_{\lambda})^{*}P
(e^{i\theta}+U^{*}_{\lambda})x\mid x\rangle -\langle (e^{i\theta}-U^{*}_{\lambda})^{*}P
(e^{i\theta}-U^{*}_{\lambda})x\mid x\rangle \bigr)\\
& \leq\frac{1}{4}\langle
(e^{i\theta}+U^{*}_{\lambda})^{*}
P(e^{i\theta}+U^{*}_{\lambda})x\mid x\rangle \\
& \leq\frac{1}{4}\left\|
(e^{i\theta}+U^{*}_{\lambda})^{*}
P(e^{i\theta}+U^{*}_{\lambda})\right\|\\
& =\frac{1}{4}\left\|P^{\frac{1}{2}}
(e^{i\theta}+U^{*}_{\lambda})
(e^{-i\theta}+U_{\lambda})P^{\frac{1}{2}}
\right\|\quad\text{(by }\|X^*X\|=\|XX^*\|)\\
& =\frac{1}{4}\left\| P+
P^{\frac{1}{2}}U^{*}_{\lambda}U_{\lambda}
P^{\frac{1}{2}}+
2\Re(e^{i\theta}
P^{\frac{1}{2}}U_{\lambda}P^{\frac{1}{2}})\right\|\\
& \leq
\frac{1}{4} \|P\|+\frac{1}{4} \|P\|\left\|U^{*}_{\lambda}U_{\lambda}\right\|+\frac{1}{2}\left\| \Re (e^{i\theta}P^{\frac{1}{2}}U_{\lambda}
P^{\frac{1}{2}})\right\|\\
&\leq
\frac{1}{4} \|P\|+\frac{1}{4} \|P\|\left\|U^{*}_{\lambda}U_{\lambda}\right\|+
\frac{1}{2}\omega
\left(P^{\frac{1}{2}}U_{\lambda}P^{\frac{1}{2}}\right)
\quad\text{(by \eqref{zamnum}).}
\end{align*}
So, by taking the supremum over all $x\in \mathcal{H}$ with $\|x\|=1$ in the above inequality and then using \eqref{zamnum} we get
\begin{align}\label{esss1}
\omega\left(U_{\lambda}P\right)
& \leq \frac{1}{4} \|P\|+\frac{1}{4} \|P\|\left\|U^{*}_{\lambda}U_{\lambda}\right\|+\frac{1}{2}\omega\left(P^{\frac{1}{2}}U_{\lambda}
P^{\frac{1}{2}}\right) \nonumber\\
 &\leq \frac{1}{4} \|P\|+\frac{1}{4} \|P\|\left\|U^{*}_{\lambda}U_{\lambda}\right\|+\frac{1}{2}\omega(\widehat{\mathbf{T}})\quad (\text{by }\; \eqref{newnumrad}).
\end{align}
On the other hand, let $x\in \mathcal{H}$ with
$\|x\|=1$ and $(\lambda_1,\ldots,\lambda_d)\in
\overline{\mathbb{B}_d}$. By applying the Cauchy-Schwarz inequality and making elementary calculations we see that
\begin{align*}
\langle U^{*}_{\lambda}U_{\lambda}x \mid x \rangle
& = \sum_{j=1}^d \sum_{k=1}^d \overline{\lambda_{j}}\lambda_{k}\langle V_kx\mid V_jx\rangle\\
&\le \sum_{j=1}^d \sum_{k=1}^d |\lambda_j|\cdot |\lambda_k| \cdot \|V_kx\|\cdot \|V_jx\|\\
&=\left( \sum_{k=1}^d |\lambda_k|\cdot \|V_kx\| \right)^2\\
%
& \le
\left( \sum_{j=1}^d |\lambda_j|^2 \right) \left( \sum_{j=1}^d \|V_jx\|^2 \right)\\
& =
\left( \sum_{j=1}^d |\lambda_j|^2 \right) \left( \sum_{j=1}^d \langle V_j^{*}V_jx |  x\rangle  \right)
\leq
\left( \sum_{j=1}^d |\lambda_j|^2 \right)\left\|\sum_{i=1}^d V_{i}^{*}V_{i}\right\|\leq 1.
\end{align*}
So, by taking the supremum over all $x\in \mathcal{H}$ with $\|x\|=1$, we obtain
$\left\|U^{*}_{\lambda}U_{\lambda}\right\|\le 1.$
This yields, by using \eqref{esss1}, that
$$
\omega\left(U_{\lambda}P\right)
 \leq \frac{1}{2} \|P\|+\frac{1}{2}\omega(\widehat{\mathbf{T}}).
$$
Thus, by taking the supremum over all $(\lambda_{1},\ldots,\lambda_{d})\in\overline{\mathbb{B}}_d$ in the above inequality and then using
\eqref{eq:numerical radius of UP}, we obtain
\begin{equation*}
\omega(\mathbf{T})\leq \frac{1}{2} \|P\|+\frac{1}{2}\omega(\widehat{\mathbf{T}}).
\end{equation*}
Therefore, we get the desired result since $\|P\|=\|\mathbf{T}\|$.
\end{proof}

\section{Spectral radius}\label{s4}
In this section, we shall characterize the joint spectral radius via spherical Aluthge transform.

\begin{theorem}\label{th:spectral radius}
Let $\mathbf{T}=(T_{1},\ldots,T_{d})\in\mathcal{B}(\mathcal{H})^{d}$ be a commuting $d$-tuple of operators. Then
$$ \lim_{n\to \infty}\|\widehat{\mathbf{T}}_{n}\|=
r(\mathbf{T}), $$
where $\widehat{\mathbf{T}}_{n}$ means the $n$-th iteration of spherical Aluthge transform, i.e.,
$\widehat{\mathbf{T}}_{n}:=
\widehat{\widehat{\mathbf{T}}_{n-1}}$, and
$ \widehat{\mathbf{T}}_{0}:=\mathbf{T}$ for a non-negative integer $n$.
\end{theorem}

We will prove this by similar arguments as in \cite{W2003}. In order to achieve the goals of the present section, we need the following results.

\begin{thqt}[\cite{bd1995}]
\label{thqt:Theorem A}
Let $A,B, X\in \mathcal{B}(\mathcal{H})$. Then
$$ \| A^{*}XB\|^2\leq \|A^{*}AX\|\,\|XBB^{*}\|. $$
\end{thqt}

\begin{thqt}[\cite{heinz}]\label{thqt:Theorem B}
Let $A,B\in \mathcal{B}(\mathcal{H})$ be positive, and
$X\in \mathcal{B}(\mathcal{H})$. Then
$$ \| A^{\alpha}XB^{\alpha}\|\leq
\|AXB\|^{\alpha}\|X\|^{1-\alpha} $$
for all $0\leq \alpha\leq 1$.
\end{thqt}

\begin{lemma}\label{lem:Lemma0}
Let ${\bf T}=(T_{1},\ldots,T_{d})\in\mathcal{B}(\mathcal{H})^{d}$ be a commuting $d$-tuple of operators. Then the spherical Aluthge transform
$\widehat{\mathbf{T}}$ is also a commuting $d$-tuple of operators.
\end{lemma}

\begin{proof}
Let $T_{k}=V_{k}P$. Then
$\widehat{\mathbf{T}}=(\widehat{T_1},\ldots,\widehat{T_d})=(P^{\frac{1}{2}}V_{1}P^{\frac{1}{2}},\ldots,
P^{\frac{1}{2}}V_{d}P^{\frac{1}{2}})$. By Lemma \ref{lem: commuting}, we have
$V_{j}PV_{k}=V_{k}PV_{j}$ for all $j,k=1,\ldots,d$.
Hence we have
$$\widehat{T_j}\widehat{T_k}=P^{\frac{1}{2}}V_{j}PV_{k}P^{\frac{1}{2}}=P^{\frac{1}{2}}V_{k}PV_{j}P^{\frac{1}{2}}
=\widehat{T_k}\widehat{T_j}.$$
\end{proof}

\begin{lemma}\label{lem:Lemma 1}
There is an $s\geq r({\bf T})$ for which
$\displaystyle \lim_{n\to \infty}
\|{\bf \hat{T}}_{n}\|=s$.
\end{lemma}

\begin{proof}
By Theorem \ref{thm:norm decreasing},
a sequence $\{ \|\widehat{\mathbf{T}}_{n}\|\}_{n=0}^{\infty}$
is decreasing, and
$$ \|\widehat{\mathbf{T}}_{n}\|\geq r(\widehat{\mathbf{T}}_{n})
=r(\mathbf{T})$$
for all non-negative integer $n$,
where the last equation is shown in  \cite{benhida}.
Hence there exists
a limit point $s$ of
 $\{ \|\widehat{\mathbf{T}}_{n}\|\}_{n=0}^{\infty}$ such that
$s\geq r(\mathbf{T})$.
\end{proof}

\begin{lemma}\label{lem:Lemma 3}
For any positive integer $k$
and non-negative integer $n$,
$$ \left\|\widehat{\mathbf{T}}_{n+1}^{k} \right\|\leq
\left\|\widehat{\mathbf{T}}_{n}^{k} \right\|.$$
\end{lemma}

\begin{proof}
Since $\widehat{\mathbf{T}}_{n+1}=\widehat{\widehat{\mathbf{T}}_{n}}$, we only prove $\left\|\widehat{\mathbf{T}}^{k}\right\|\leq \left\|\mathbf{T}^{k}\right\|$.
We notice that by Lemma
\ref{lem:norm of tuple of operators},
$\|\mathbf{T}^{k}\|$ is given as follows.
\begin{equation*}
\left\|\mathbf{T}^{k}\right\|^{2}=
\left\| \sum_{i_{1},\cdots,i_{k}=1}^{d}
T_{i_{1}}^{*}\cdots T_{i_{k}}^{*}
T_{i_{k}}\cdots T_{i_{1}}\right\|.
\end{equation*}
Let $A_{k}:={\text{diag}}(P,\ldots,P)$ be a $d^{k}$--by--$d^{k}$ operator matrix, and let
$$ X_{k}=\begin{pmatrix}
V_{1}P\cdots PV_{1} & 0 & \cdots & 0 \\
\vdots & \vdots & & \vdots \\
V_{d}P\cdots PV_{d} & 0 & \cdots & 0
\end{pmatrix}$$
be a $d^{k}$--by--$d^{k}$ operator matrix,
where the $1$st column contains
$V_{i_{1}}PV_{i_{2}}P\cdots PV_{i_{k}}$ for
all $i_{1},\ldots,i_{k}=1,2,\ldots,d$.
Then by Theorem \ref{thqt:Theorem A},
\begin{align}
\left\|\widehat{\mathbf{T}}^{k}\right\|^{2}
& =
\left\| \sum_{i_{1},\ldots,i_{k}=1}^{d}
\widehat{T_{i_1}}^{*}\cdots \widehat{T_{i_k}}^{*}
\widehat{T_{i_k}}\cdots \widehat{T_{i_1}}
\right\| \nonumber\\
& =
\left\| \sum_{i_{1},\ldots,i_{k}=1}^{d}
P^{\frac{1}{2}}V_{i_{1}}^{*}P\cdots PV_{i_{k}}^{*}
PV_{i_{k}}P\cdots PV_{i_{1}}P^{\frac{1}{2}}\right\|
\nonumber \\
& =
\left\| A_{k}^{\frac{1}{2}}X_{k}^{*}A_{k}X_{k}
A_{k}^{\frac{1}{2}}\right\|
 =
\left\| A_{k}^{\frac{1}{2}}X_{k}
A_{k}^{\frac{1}{2}}\right\|^{2}
\leq
\|A_{k}X_{k}\| \|X_{k}A_{k}\|.
\label{eq: (**)}
\end{align}
Now, it can be seen that
\begin{align}
\|A_{k}X_{k}\| & =
\|X_{k}^{*}A_{k}^{2}X_{k}
\|^{\frac{1}{2}} \nonumber\\
& =
\left\|
\sum_{i_{1},\ldots,i_{k}=1}^{d}
V_{i_{1}}^{*}P\cdots PV_{i_{k}}^{*}
P^{2} V_{i_{k}}P\cdots PV_{i_{1}}
\right\|^{\frac{1}{2}}
\nonumber\\
& =
\left\|
\sum_{i_{1},\ldots,i_{k}=1}^{d}
V_{i_{1}}^{*}P\cdots PV_{i_{k}}^{*}
P\left( \sum_{i_{k+1}=1}^{d}V_{i_{k+1}}^{*}V_{i_{k+1}}
\right)P V_{i_{k}}P\cdots PV_{i_{1}}
\right\|^{\frac{1}{2}}
\nonumber\\
& =
\left\|
\sum_{i_{1}=1}^{d} V_{i_{1}}^{*}
\left(
\sum_{i_{2},\ldots,i_{k+1}=1}^{d}
PV_{i_{2}}^{*}P\cdots PV_{i_{k}}^{*}
PV_{i_{k+1}}^{*}V_{i_{k+1}} P V_{i_{k}}
P\cdots PV_{i_{2}} P\right) V_{i_{1}}
\right\|^{\frac{1}{2}}
\nonumber\\
& =
\left\|
\sum_{i_{1}=1}^{d} V_{i_{1}}^{*}
\left(
\sum_{i_{2},\ldots,i_{k+1}=1}^{d}
T_{i_{2}}^{*}\cdots T_{i_{k}}^{*}
T_{i_{k+1}}^{*}T_{i_{k+1}} \cdots T_{i_{2}} \right)
V_{i_{1}}
\right\|^{\frac{1}{2}}
\nonumber\\
& \leq
\left\| \sum_{i_{1}=1}^{d}V_{i_1}^{*}V_{i_1}
\right\|^{\frac{1}{2}}
\left\| \sum_{i_{2},\ldots,i_{k+1}=1}^{d}
T_{i_{2}}^{*}\cdots T_{i_{k+1}}^{*}
T_{i_{k+1}}\cdots T_{i_{2}}\right\|^{\frac{1}{2}}
=\|\mathbf{T}^{k}\|,\label{eq:starstar}
\end{align}
where the last inequality follows from Lemma \ref{lem:Lemma 2} and the fact that $\sum_{k=1}^{d}V_{k}^{*}V_{k}$ is a projection
onto $\overline{\mathcal{R}(P)}$. Moreover
\begin{align*}
\|X_{k}A_{k}\| & =
\|A_{k}X_{k}^{*}X_{k}A_{k}\|^{\frac{1}{2}} \\
& =
\left\| \sum_{i_{1},\ldots,i_{k}=1}^{d}
PV_{i_1}^{*}P\cdots PV_{i_{k}}^{*}
V_{i_{k}}P\cdots PV_{i_1}P\right\|^{\frac{1}{2}} \\
& =
\left\| \sum_{i_{1},\ldots,i_{k}=1}^{d}
T_{i_{1}}^{*}\cdots T_{i_{k}}^{*}
T_{i_{k}}\cdots T_{i_{1}}\right\|^{\frac{1}{2}} =
\|\mathbf{T}^{k}\|.
\end{align*}
Hence we have
$$\left\|\widehat{\mathbf{T}}^{k}\right\|\leq
\|A_{k}X_{k}\|^{\frac{1}{2}} \|X_{k}A_{k}\|^{\frac{1}{2}}
\leq \left\|\mathbf{T}^{k}\right\|.$$
\end{proof}

\begin{lemma}\label{lem:Lemma 4}
For any positive integer $k$,
$$ \left\|\widehat{\mathbf{T}}_{n+1}^{k}\right\|\leq
\left\|\widehat{\mathbf{T}}_{n}^{k+1}\right\|^{\frac{1}{2}}
\left\|\widehat{\mathbf{T}}_{n}^{k-1}\right\|^{\frac{1}{2}} $$
for all $n\geq 0$.
\end{lemma}

\begin{proof}
We shall prove
$ \|\widehat{\mathbf{T}}^{k}\|\leq
\|\mathbf{T}^{k+1}\|^{\frac{1}{2}}
\|\mathbf{T}^{k-1}\|^{\frac{1}{2}}. $
Let $A_{k}$ and $X_{k}$ be defined in the proof of
Lemma \ref{lem:Lemma 3}.
Then, by \eqref{eq: (**)} and
Theorem \ref{thqt:Theorem B}, we have
$$ \left\|\widehat{\mathbf{T}}^{k}\right\|=
\left\|A_{k}^{\frac{1}{2}}X_{k}A^{\frac{1}{2}}_{k}\right\|\leq
\|A_{k}X_{k}A_{k}\|^{\frac{1}{2}}\|X_{k}\|^{\frac{1}{2}}. $$
By taking into consideration the fact that $\sum_{k=1}^{d}V_{k}^{*}V_{k}$ is an orthogonal projection
onto $\overline{\mathcal{R}(P)}$, it can be observed that
\begin{align*}
\|A_{k}X_{k}A_{k}\|
& =
\left\| \sum_{i_1,\ldots,i_k=1}^{d}
P V_{i_1}^{*}P\cdots PV_{i_k}^{*}P^{2}
V_{i_k}P\cdots P V_{i_1}P\right\|^{\frac{1}{2}} \\
& =
\left\| \sum_{i_1,\ldots,i_k=1}^{d}
P V_{i_1}^{*}P\cdots PV_{i_k}^{*}P
\left(\sum_{i_{k+1}=1}^{d}V_{i_{k+1}}^{*}V_{i_{k+1}}\right)
PV_{i_k}P\cdots P V_{i_1}P\right\|^{\frac{1}{2}} \\
& =
\left\| \sum_{i_1,\ldots,i_{k+1}=1}^{d}
P V_{i_1}^{*}P\cdots PV_{i_k}^{*}P
V_{i_{k+1}}^{*}V_{i_{k+1}}
PV_{i_k}P\cdots P V_{i_1}P\right\|^{\frac{1}{2}} \\
& =
\left\| \sum_{i_1,\ldots,i_{k+1}=1}^{d}
T_{i_1}^{*}\cdots T_{i_{k+1}}^{*}T_{i_{k+1}}
\cdots T_{i_1}\right\|^{\frac{1}{2}}
 =\left\|\mathbf{T}^{k+1}\right\|.
\end{align*}
On the other hand, one has
\begin{align*}
\|X_{k}\|
& =
\left\| \sum_{i_{1},\ldots,i_{k}=1}^{d}
V_{i_{1}}^{*}P\cdots PV_{i_{k}}^{*}
V_{i_{k}}P\cdots PV_{i_{1}}\right\|^{\frac{1}{2}}\\
& =
\left\| \sum_{i_{1},\ldots,i_{k-1}=1}^{d}
V_{i_{1}}^{*}P\cdots P
\left(\sum_{i_{k}=1}^{d}
V_{i_{k}}^{*}V_{i_{k}}\right)
P\cdots PV_{i_{1}}\right\|^{\frac{1}{2}}\\
& =
\left\| \sum_{i_{1},\ldots,i_{k-1}=1}^{d}
V_{i_{1}}^{*}P\cdots V_{i_{k-1}}^{*}P^{2}V_{i_{k-1}}
\cdots PV_{i_{1}}\right\|^{\frac{1}{2}}\\
& =
\left\| X_{k-1}^{*}A_{k-1}^{2}X_{k-1}
\right\|^{\frac{1}{2}}
\leq \left\|\mathbf{T}^{k-1}\right\|,
\end{align*}
where the last inequality follows from
\eqref{eq:starstar}.
Therefore
$$ \|\widehat{\mathbf{T}}^{k}\|\leq
\|A_{k}X_{k}A_{k}\|^{\frac{1}{2}}
\|X_{k}\|^{\frac{1}{2}}
\leq \|\mathbf{T}^{k+1}\|^{\frac{1}{2}}
\|\mathbf{T}^{k-1}\|^{\frac{1}{2}}.$$
\end{proof}

\begin{lemma}\label{lem:Lemma 5}
For each positive integer $k$,
$\|\mathbf{T}^{k+1}\|\leq
\|\mathbf{T}^{k}\| \|\mathbf{T}\|. $
\end{lemma}

\begin{proof}
\begin{align*}
\|\mathbf{T}^{k+1}\|^{2}
& =
\left\| \sum_{i_{1},\ldots,i_{k+1}=1}^{d}
T_{i_1}^{*}\cdots T_{i_{k+1}}^{*}
T_{i_{k+1}}\cdots T_{i_1}\right\| \\
& =
\left\| \sum_{i_{1}=1}^{d}T_{i_1}^{*}\left(
\sum_{i_{2},\ldots,i_{k+1}=1}^{d}
T_{i_2}^{*}\cdots T_{i_{k+1}}^{*}
T_{i_{k+1}}\cdots T_{i_2}\right)T_{i_1}\right\| \\
& \leq
\left\| \sum_{i_{1}=1}^{d} T_{i_{1}}^{*}T_{i_{1}}\right\|
\left\| \sum_{i_{2},\ldots,i_{k+1}=1}^{d}
T_{i_2}^{*}\cdots T_{i_{k+1}}^{*}
T_{i_{k+1}}\cdots T_{i_2}\right\|
\quad\text{(by Lemma \ref{lem:Lemma 2})}\\
& =
\|\mathbf{T}\|^{2}\|\mathbf{T}^{k}\|^{2}.
\end{align*}
\end{proof}

\begin{lemma}\label{lem:Lemma 6}
For any positive integer $k$,
$\displaystyle \lim_{n\to\infty}
\|{\bf \hat{T}}_{n}^{k}\|=s^{k}$.
\end{lemma}

\begin{proof}
We will prove the lemma by induction.
Since $\displaystyle \lim_{n\to\infty}
\|{\bf \hat{T}}_{n}\|=s$ by Lemma \ref{lem:Lemma 1},
the lemma is proven for $k=1$.
Assume the lemma is proven for $1\leq k\leq m$.
By Lemmas \ref{lem:Lemma 4} and
\ref{lem:Lemma 5},
\begin{equation}
\begin{split}
\|\widehat{\mathbf{T}}_{n+1}^{k}\|
& \leq
\|\widehat{\mathbf{T}}_{n}^{k+1}\|^{\frac{1}{2}}
\|\widehat{\mathbf{T}}_{n}^{k-1}\|^{\frac{1}{2}}\\
& \leq
\|\widehat{\mathbf{T}}_{n}^{k}\|^{\frac{1}{2}}
\|\widehat{\mathbf{T}}_{n}\|^{\frac{1}{2}}
\|\widehat{\mathbf{T}}_{n}^{k-1}\|^{\frac{1}{2}}.
\end{split}
\label{eq:(*3)}
\end{equation}
Let $\displaystyle t:=\lim_{n\to\infty}
\|\widehat{\mathbf{T}}_{n}^{m+1}\|$.
The existence of limit
follows from Lemma \ref{lem:Lemma 3}.
Taking limits, the induction hypothesis and
\eqref{eq:(*3)} show that
$$ s^{m}\leq t^{\frac{1}{2}}s^{\frac{m-1}{2}}
\leq s^{\frac{m}{2}}s^{\frac{1}{2}}s^{\frac{m-1}{2}}=
s^{m}. $$
It follows that $t=s^{m+1}$, and the proof is
completed.
\end{proof}

\begin{proof}[Proof of Theorem \ref{th:spectral radius}]
It follows from Lemmas \ref{lem:Lemma 3} and
\ref{lem:Lemma 6} that, for each positive
integer $k$, the decreasing sequence
$\{ \|\widehat{\mathbf{T}}_{n}^{k}\|^{\frac{1}{k}}\}_{n=0}^{\infty}$
converges to $s$. Therefore
\begin{equation}
s\leq \| \widehat{\mathbf{T}}_{n}^{k}\|^{\frac{1}{k}}
\label{eq:(*4)}
\end{equation}
for all $n$ and $k$. Now fix an $n$.
If $r({\bf T})<s$, then by
Lemma \ref{lem:Lemma0} and \eqref{jointradii002},
$$ \lim_{k\to\infty}\|\mathbf{T}^{k}\|^{\frac{1}{k}}
=r(\widehat{\mathbf{T}_{n}})=r({\bf T}) $$
would imply that
$\displaystyle \|\widehat{\mathbf{T}}_{n}^{k}
\|^{\frac{1}{k}}<s$
for sufficiently large $k$. Clearly this is a
contradiction to \eqref{eq:(*4)}.
Therefore, we must have $s=r({\bf T})$, and the
result follows from Lemma \ref{lem:Lemma 1}.
\end{proof}

\begin{remark}
For a $d$-tuple of operators $\mathbf{T}$
and a natural number $n$,
$\mathbf{T}^{n}$ is a $d^{n}$-tuple of operators.
Then we should consider
$d^{n}$-tuple of operators for $n=1,2,\ldots$
to use \eqref{jointradii002}.
However, since $\widehat{\mathbf{T}}_{n}$ is also a $d$-tuple
of operators, we only treat
$d$-tuple of operators to get $r(\mathbf{T})$ by
Theorem \ref{th:spectral radius}.
\end{remark}


\end{document}